\documentclass{amsart}
\usepackage{amsmath}
\usepackage{amssymb}
\usepackage{color}

\DeclareMathOperator{\Id}{Id}

\def\FF{\mathbb{F}}

\newtheorem{theorem}{Theorem}[section]
\newtheorem{lemma}[theorem]{Lemma}
\newtheorem{proposition}[theorem]{Proposition}
\newtheorem{corollary}[theorem]{Corollary}
\theoremstyle{definition}
\newtheorem{definition}[theorem]{Definition}
\newtheorem{example}[theorem]{Example}

\theoremstyle{remark}
\newtheorem{remark}[theorem]{Remark}
\numberwithin{equation}{section}

\def\separa{\hbox to 14 truecm{\hrulefill}}

\author[P. Danchev]{Peter Danchev}
\address{Institute of Mathematics and Informatics, Bulgarian Academy of Sciences, 1113 Sofia, Bulgaria}
\email{danchev@math.bas.bg}
\thanks{The first author was partially supported by the Bulgarian National Science Fund under Grant KP-06 No. 32/1 of December 07, 2019.}

\author[E. Garc\'\i a]{Esther Garc\'\i a}
\address{ Departamento de Matem\'{a}tica  Aplicada, Ciencia e Ingenier\'{\i}a de los Materiales y Tecnolog\'{\i}a Electr\'onica,
Universidad Rey Juan Carlos, 28933 M\'{o}s\-to\-les (Madrid), Spain}
\email{esther.garcia@urjc.es}
\thanks{The second author was partially supported by Ayuda Puente 2022, URJC}

\author[M. G\'omez Lozano]{Miguel G\'omez Lozano}
\address{Departamento de \'Algebra, Geometr\'{\i}a y
Topolog\'{\i}a, Universidad de M\'alaga, 29071 M\'alaga, Spain}
\thanks{The three authors were partially supported by  the Junta de Andaluc\'{\i}a FQM264.}
\email{miggl@uma.es}

\begin{document}

\title[Decompositions of Matrices Into a Sum of Torsion and Nilpotent Matrices]{Decompositions of Matrices Into a Sum of
\\ Torsion Matrices and Matrices of Fixed Nilpotence}
\maketitle

\begin{abstract}
For $n\ge 2$ and  fixed   $k\ge 1$,  we  study when a square matrix  $A$  over an arbitrary field $\mathbb{F}$ can be decomposed as $T+N$ where $T$ is a torsion matrix and $N$ is a nilpotent matrix with $N^k=0$. For fields of prime characteristic, we show that this  decomposition holds as soon as the characteristic polynomial of $A\in \mathbb{M}_{n}(\mathbb{F})$ is algebraic over its base field and the rank of $A$ is at least $\frac nk$, and we present several examples that show that the decomposition does not hold in general. Furthermore, we completely solve this decomposition problem for $k=2$ and nilpotent matrices over arbitrary fields (even over division rings). This somewhat continues our recent publications in Lin. \& Multilin. Algebra (2023) and Internat. J. Algebra \& Computat. (2022) as well as it strengthens results due to Calugareanu-Lam in J. Algebra \& Appl. (2016).
\end{abstract}



\section{Introduction and Fundamentals}

The decomposition of matrices over an arbitrary field into the sum of some special elements, like nilpotent elements, idempotent elements, potent elements, units, etc., was in the focus of so many researchers for a long time (see, e.g., \cite{AM}, \cite{AT1}, \cite{AT2}, \cite{B}, \cite{BCDM}, \cite{BM}, \cite{Sh}, \cite{St1} and \cite{St2} and the bibliography cited therewith). Specifically, concerning our own work on the subject, in \cite{DGL1} we found some necessary (and sufficient) conditions when any square matrix over a field (finite or infinite) is expressible as a sum of a diagonalizable matrix and a nilpotent matrix of index less than or equal to two. In particular, we also obtained some results on the expression of square matrices into the sum of a potent matrix and a square-zero matrix over finite fields. Nevertheless, such a decomposition does not hold for fields of zero characteristic (see \cite[Example 4.3]{DGL1}). Further insight in that matter over some special finite rings was achieved by us in \cite{DGL2}. We also refer the interested reader to \cite{D} for some other aspects of the realization of matrices into the sum of specific elements over certain fields.

%

By combining the notions of invertibility and nilpotence, C\v{a}lug\v{a}reanu and  Lam introduced in 2016 the notion of {fine} rings \cite{CL}: those in which every nonzero element can be written as the sum of an invertible element and a nilpotent one, proving in that work that every nonzero square matrix over a division ring is the sum of an invertible matrix and a nilpotent matrix. The rings whose nonzero idempotents are fine turned out to be an interesting class of indecomposable rings and were studied in \cite{CZ1} by  C\v{a}lug\v{a}reanu and Zhou.  In 2021, the same authors focused on rings in which every nonzero nilpotent element is fine, which they called $NF$ rings, and showed that for a commutative ring $R$ and $n\ge 2$, the matrix ring $\mathbb{M}_n(R)$ is $NF$ if and only if $R$ is a field; see \cite{CZ}. A slightly more general class of rings than  fine rings was defined in \cite{Da} under the name {nil-good rings} (every element $a$ can be expressed the sum $a=n+u$ where $n$ is  nilpotent and  $u$ is either zero or a unit); in \cite{GS} it is shown that the matrix ring $\mathbb{M}_{n}(D)$ over a division ring $D$ is nil-good.  In general, no restriction in the index of nilpotence is required in  these decompositions.

In our work \cite{DGL3} we considered the ring of matrices $\mathbb{M}_{n}(\mathbb{F})$ over an arbitary field $\mathbb{F}$, we fixed a bound $k$ for the index of nilpotence, and studied when a  matrix $A$ in  $\mathbb{M}_{n}(\mathbb{F})$ could be expressed as the sum of an invertible matrix $U$ and a nilpotent matrix $N$ with $N^k=0$. Here we will continue our study in this branch by replacing the invertibility condition of $U$ by being a torsion matrix.  Recall that a   torsion matrix $T$ is the one for which there is a positive integer $s$ such that $T^s$ is the identity matrix. One elementarily sees that such a matrix is necessarily invertible as well as that it is $s+1$-potent, i.e., $T^{s+1}=T$. Canonical forms of torsion matrices were studied by D. Sjerve and a full classification over the rational numbers is presented in his paper \cite{Sj}.

The paper is organized as follows: in the first section we will show that the desired  decomposition holds as soon as the characteristic polynomial of $A$ is algebraic over its base field and its rank satisfies a certain bound, and we present several examples that show that the decomposition does not hold in general. In the second section, we focus on nilpotent matrices and deal with the problem of finding a necessary and sufficient condition to decompose such matrices as the sum of a torsion matrix and a zero-square matrix (fixed nilpotence $k\le 2$). Since we solve this problem by dealing with the Jordan canonical blocks of the considered nilpotent matrix, our result also holds for nilpotent matrices over division rings.

\section{Decomposing Matrices Into a Sum of Torsion Matrices and Matrices of Fixed Nilpotence}

As usual, for convenience of the presentation, let us specify that the letter $\mathbb{F}$ will stand an arbitrary field unless it is not specified something else, and the symbol $\mathbb{M}_n(\mathbb{F})$ is reserved for the matrix ring over $\mathbb{F}$. All other unexplained explicitly notations are standard and will be in an agreement with the book \cite{H}.

\medskip

In our work \cite{DGL3} we showed:

\begin{theorem}\label{nilpotentmain}\cite[Theorem 2.7]{DGL3}
Let $n\ge 2$. Let $\mathbb{F}$ be a field, consider the ring $\mathbb{M}_{n}(\mathbb{F})$ and let us fix $k\ge 1$. Given a nonzero matrix $A\in \mathbb{M}_{n}(\mathbb{F})$,  there exists an invertible matrix $U\in \mathbb{M}_{n}(\mathbb{F})$ and a nilpotent matrix $N\in \mathbb{M}_{n}(\mathbb{F})$ with $N^k=0$ such that $A=U+N$ if and only if the rank of $A$ is greater than or equal to $\frac{n}{k}$
\end{theorem}

In this section, we will address the following query:

\medskip

\noindent{\bf Problem:} {\it Given a fixed  $k\ge 1$, find necessary and sufficient conditions to decompose any non-zero square matrix $A$ over a field $\mathbb{F}$ as a sum of a torsion matrix $T$ and a  nilpotent matrix $N$ with $N^k=0$.}

\medskip

Notice that the proposed Problem is already solved for matrices over finite fields  by using Theorem~\ref{nilpotentmain} and the obvious fact that the unit group of finite rings is always torsion. Nevertheless, the rank condition is not enough to guarantee this decomposition when working over infinite fields. In the rest of this section, we will show some cases when this decomposition holds, and some counterexamples showing that the decomposition does not hold in general.

\begin{remark}\label{trace}
Let $A\in \mathbb{M}_n(\mathbb{F})$. If there exists a nilpotent matrix $N\in \mathbb{M}_n(\mathbb{F})$ ($N^k=0$) and $T=A-N$ satisfies $T^s=\Id$ for some $s\in \mathbb{N}$, then the following three points are fulfilled:

\begin{itemize}
\item the trace of $A$ coincides with the trace of $T$;
\item the minimal polynomial of $T$ divides $X^s-1$ and therefore the eigenvalues of $T$ (in some extension of $\mathbb{F}$) are $s$-roots of the unity. Moreover,  if  $X^s-1$ is separable, $T$ is diagonalizable;
\item  the trace of $T$ coincides with the sum of its eigenvalues, so it is an algebraic number over the base field of $\mathbb{F}$.
\end{itemize}

For example, a matrix $A\in \mathbb{M}_n(\mathbb{F})$, even of full rank, and whose trace is transcendent over its base field can never be decomposed into the sum $T+N$, where $T$ is a torsion matrix and $N$ is a nilpotent matrix.
\end{remark}

\medskip

Let $n\ge 2$. Recall that the trace of a polynomial $p(x)=x^{n}+b_{n-1}x^{n-1}+\dots+b_0\in \mathbb{F}[x]$ is the scalar $-b_{n-1}$ and coincides with the trace of the companion matrix $C(p(x))\in\mathbb{M}_n(\mathbb{F})$. Notice that the rank of a companion matrix is always greater than or equal to $n-1$.

We can now give a partial solution to the proposed above Problem. Concretely, the following statements hold.

\begin{proposition}\label{basecasetorsion} Let $n\ge 2$, let $p(x)\in \mathbb{F}[x]$ be a polynomial of degree $n$ and let $C(p(x))\in\mathbb{M}_n(\mathbb{F})$ be its companion matrix. If the trace of $p(x)$ can be expressed as the sum of $n$ different roots of the unity in some extension of $\mathbb{F}$, then $A$ can be decomposed (in some extension of $\mathbb{F}$) into $T+N$, where $T$ is a torsion matrix and $N^2=0$. In particular, this always holds if the trace of $p(x)$ is either $1$, or $-1$, or $0$.
\end{proposition}

\begin{proof}
By hypothesis, the trace of $p(x)=x^{n}+b_{n-1}x^{n-1}+\dots+b_0$ can be expressed as $\alpha_1+\dots+\alpha_{n}$ for some different roots of unity  $\alpha_1,\dots,\alpha_{n}$ in some extension of $\mathbb{F}$.
Let us consider the polynomial $q(x)=(x-\alpha_1)\cdots(x-\alpha_{n})=x^{n}+a_{n-1}x^{n-1}+\dots+a_0$. Thus, we have
\begin{align*}
C&(p(x))=\left(
          \begin{array}{cccc}
            0 & 0 & \dots  & -b_0 \\
            1 & 0 &  &  \vdots\\
             & \ddots & \ddots &  \\
            0 &  & 1 & -b_{n-1}\\
          \end{array}
        \right)\\&=
        \underbrace{\left(
          \begin{array}{cccc}
            0 & 0 & \dots  & -a_0 \\
            1 & 0 &  &  \vdots\\
             & \ddots & \ddots &  \\
            0 &  & 1 & -a_{n-1}\\
          \end{array}
        \right)}_{C(q(x))}+\underbrace{\left(
          \begin{array}{cccc}
            0 & 0 & \dots  & a_0-b_0 \\
            0 & 0 &  &  \vdots\\
             & \ddots & \ddots &  \\
            0 &  & 0 & a_{n-1}-b_{n-1}\\
          \end{array}
        \right)}_{N},
\end{align*}
where $T:=C(q(x))$ is a diagonalizable torsion matrix and $N^2=0$,
because $a_{n-1}-b_{n-1}=0$, as required. In view of these arguments, the last claim follows now at once.
\end{proof}

As an immediate consequence, we obtain:

\begin{corollary} Let $A\in\mathbb{M}_n(\mathbb{F})$ and let $p_1(x),\dots, p_k(x)$ be the elementary divisors of $A$ (respectively, the invariant factors of $A$). If each $p_i(x)$ has degree $n_i$ and its trace is a sum of $n_i's$ different roots of the unity in some extension of $\mathbb{F}$, then $A$ can be decomposed (in some extension of $\mathbb{F}$) into $T+N$, where $T$ is a torsion matrix and $N^2=0$.
\end{corollary}

\begin{proof} If $p_1(x),\dots, p_k(x)$ are the elementary divisors (respectively, the invariant factors of $A$), then the matrix $A$ is similar to a direct sum of companion matrices of each $p_i(x)$. However, utilizing Proposition \ref{basecasetorsion}, we can express every companion matrix of $p_i(x)$ into the sum $T_i+N_i$, where $T_i$ is a torsion matrix and $N_i^2=0$, as needed.
\end{proof}

The next two curious comments are worthwhile.

\begin{remark}\label{counterexamplenxn}
Not every matrix $A\in \mathbb{M}_n(\mathbb{F})$ whose trace is a sum of $n$ roots of unity can be written as $T+N$, where $T$ is a torsion matrix and $N^2=0$ (even if those roots are different and $A$ satisfies the rank condition). Indeed, let $n\ge 2$ and let us consider the element $a\in \mathbb{F}$ such that $na$ is a sum of roots of unity, but $a$ itself is not a root of unity (notice that such an element $a$ always exists and is easy to be constructed, so we leave out the details). Then the matrix $A=a \Id\in \mathbb{M}_n(\mathbb{F})$ cannot be written as a sum $T+N$; otherwise there would exist $m\in\mathbb{N}$ such that $T^m=\Id$; but then $\Id=T^m=(A-N)^m=a^m \Id - ma^{m-1}N$, so that $(a^m-1)^2\Id=((a^m-1)\Id)^2=(m a^{m-1}N)^2=0$, a contradiction.\\
For example, the matrix $A=\left(
                             \begin{array}{cc}
                               \frac 12 & 0 \\
                               0 & \frac 12 \\
                             \end{array}
                           \right)
$
cannot be expressed as $T+N$ even if its trace is the sum of two (different) $6^{\rm th}$-roots of unity: $\frac 12+ \frac{\sqrt{3}}{2}i$ and $\frac 12- \frac{\sqrt{3}}{2}i$.
\end{remark}

\begin{remark}
Let $\mathbb{F}$ be a field of characteristic 0.
When $n\ge 3$, if $p(x)$ is a polynomial of degree $n$ and the trace of $p(x)$  is the sum of $n$ equal roots of unity, then the matrix $C(p(x))$  can never be written as $T+N$.
Indeed, let $\alpha$ be a root of unity, and consider a degree $n$ polynomial $p(x)$  whose trace is $n\alpha$ and its companion  matrix $C(p(x))$. Suppose now that $C(p(x))=T+N$ where $T$ satisfies $T^m=\Id$ for some $m\in \mathbb{N}$ and $N^2=0$. Since the minimal polynomial of $T$ divides $X^m-1$ and this polynomial has no multiple roots, then $T$ is diagonalizable and its eigenvalues are all roots of unity whose sum coincides with the trace of $T$ (which, on the other side, coincides with the trace of $p(x)$), so it is exactly $n\alpha$.
The only solution to $n\alpha=\alpha_1+\dots+\alpha_n$, $\alpha$, $\alpha_1$, $\dots$, $\alpha_n$ being roots of unity, is $\alpha=\alpha_1=\dots=\alpha_n$.
 Therefore, the eigenvalues of $T$ are all equal to $\alpha$ and thus $T=\alpha\Id$. But then $N=C(p(x))-\alpha\Id$ should have zero square, which is manifestly untrue (the rank of $C(p(x))-\alpha\Id$ is at least $n-1$).

For example, the matrix  $$
C((x-1)^3)=\left(
             \begin{array}{ccc}
               0 & 0 & 1 \\
               1 & 0 & -3 \\
               0 & 1 & 3 \\
             \end{array}
           \right)
$$
cannot be expressed as $T+N$, where $T$ is a torsion matrix and $N^2=0$.
\end{remark}

\medskip

Nevertheless, when we focus on fields of prime characteristic, we can partially solve the Problem. Let $\mathbb{F}$ be a field of prime characteristic $p$ and let us denote by $\mathbb{F}_p$ its base field.

\begin{lemma}\label{Th1} Let $n\ge 2$.
If $\mathbb{F}$ is a field of characteristic $p$ and we fix  $k\ge1$, then for every matrix $A\in \mathbb{M}_n(\mathbb{F})$ of rank greater than or equal to $\frac nk$ and whose entries are algebraic over $\mathbb{F}_p$ there exists a matrix $N$ with $N^k=0$ such that $A-N$ is a torsion matrix.
\end{lemma}

\begin{proof}
The subfield $\mathbb{K}$ of $\mathbb{F}$ generated by the base field $\mathbb{F}_p$ and by the entries of matrix $A$ is a finite field, and $A\in  \mathbb{M}_n(\mathbb{K})$. Apply Theorem \ref{nilpotentmain} to decompose $A$ as $U+N$, where $U\in \mathbb{M}_n(\mathbb{K})$ is invertible and $N\in \mathbb{M}_n(\mathbb{K})$ satisfies $N^k=0$. Since $U$ is an invertible matrix over a finite field, being invertible is equivalent to being a torsion matrix, as wanted.
\end{proof}

The conditions on the entries of the matrix can be translated to the coefficients of the characteristic polynomial of the matrix $A$. We will say that a polynomial is algebraic over $\mathbb{F}_p$ is all its coefficients are algebraic over $\mathbb{F}_p$.

\begin{theorem}\label{nilpotentprimefield}
Let  $\mathbb{F}$ be a field of characteristic $p$, let us fix an index of nilpotence $k\ge 1$ and let $A\in \mathbb{M}_n(\mathbb{F})$ of rank greater than or equal to $\frac nk$. If the characteristic polynomial of $A$ is algebraic over $\mathbb{F}_p$, then $A$ can be written as $T+N$, where $T$ is a torsion matrix and $N^k=0$. In particular, this decomposition always holds for nilpotent matrices of rank greater than or equal to $\frac nk$.
\end{theorem}

\begin{proof}
Let us consider the primary rational canonical form $C$ of $A$, whose characteristic polynomial is algebraic over $\mathbb{F}_p$. The eigenvalues of $C$ (roots in some extension of $\mathbb{F}$ of the characteristic polynomial) are algebraic over $\mathbb{F}_p$ and, therefore, all the elementary divisors of $A$ are algebraic polynomials over $\mathbb{F}_p$. Consequently, the entries of $C$ are all algebraic over $\mathbb{F}_p$ and we can apply Lemma \ref{Th1} to get the proof.
\end{proof}

\medskip

\noindent{\bf Open Question 1:} Given a fixed index of nilpotence $k\ge 1$, find a suitable criterion for the decomposition of an arbitrary matrix over a field of zero characteristic into the sum of a torsion matrix and a  nilpotent matrix of index of nilpotence $\le k$.

\medskip

In the following section we will answer this question for $k\le 2$ and nilpotent matrices of rank at least $\frac n2$. Since our arguments are quite technical, we leave open the question of decomposing nilpotent matrices of rank at least $\frac nk$ into torsion and nilpotent matrices of index less than or equal to $k$.

\section{Decomposing Nilpotent Matrices Into a Sum of Torsion and Square-Zero Matrices}

The goal of this  section is to show that, for any field $\FF$, every nilpotent matrix in $\mathbb{M}_n(\mathbb{F})$ whose rank is at least $\frac n2$ can always be decomposed as the sum of a torsion matrix and a square-zero matrix. Recall that every nilpotent matrix is similar to a direct sum of its Jordan blocks -- all of them associated to the eigenvalue 0.

As in the previous section, the main difficulties arise when dealing with Jordan blocks of size 1. The condition on the rank guarantees that they can always be combined with Jordan blocks of bigger size. Let us consider the following two examples.

\begin{example}
Let $A$ be the nilpotent matrix

\medskip

$$
  A=\left(
      \begin{array}{ccc|c}
        0 & 0 & 0 & 0 \\
        1 &0 & 0 & 0 \\
        0 & 1 & 0 & 0 \\
        \hline
        0 & 0 & 0 & 0 \\
      \end{array}
    \right),
$$

\medskip

\noindent which is the smallest nilpotent matrix decomposed into Jordan blocks with one Jordan block of size 1 and  rank of $A$ greater than or equal to $\frac n2$ ($n=4$). Let us define

\medskip

$$
  N=\left(
      \begin{array}{cccc}
        0 &1 & 0 & -1 \\
       0 & 0 & -1 & 0 \\
        0 & 0 & 0 & 0 \\
       0 & 0 & -1 & 0 \\
      \end{array}
    \right)
$$

\medskip

\noindent and let $T=A-N$. Then $N^2=0$ and the characteristic polynomial of $T$ is $x^4-1$, so $A=T+N$ with $T^4=\Id$ and $N^2=0$.
\end{example}

\begin{example} In some situations, more than one Jordan block of size 1 has to be combined with a Jordan block of bigger size. Consider for example
$$
A=\left(
    \begin{array}{cccc|cc}
      0 & 0 & 0 & 0 &0 & 0 \\
     1 & 0 & 0 & 0 &0 & 0 \\
     0 & 1& 0 & 0 &0 & 0 \\
      0 & 0 & 1& 0 &0 & 0 \\
      \hline
      0 & 0 & 0 & 0 &0 & 0 \\
      0 & 0 & 0 & 0 &0 & 0 \\
    \end{array}
  \right),
$$
which again is the smallest nilpotent matrix expressed in Jordan blocks having two blocks of size 1 and satisfying ${\rm rank}(A)\ge \frac n2$ ($n=6$ here). Let us consider the zero-square matrix
$$
N=\begin{pmatrix}0 & 0 & 1 & 0 & 0 & -1\\
0 & 0 & 0 & -1 & 0 & 0\\
0 & 1 & 0 & 0 & -1 & 0\\
0 & 0 & 0 & 0 & 0 & 0\\
0 & 0 & 0 & -1 & 0 & 0\\
0 & 1 & 0 & 0 & -1 & 0\end{pmatrix}.
$$
Then $T=A-N$ has characteristic polynomial equal to $x^6-1$, so $A=T+N$ with $T^6=\Id$ and $N^2=0$.
\end{example}
The construction of the zero-square matrices above depend on how many Jordan blocks of size 1 we have to combine with a Jordan block of bigger size. Let $s$ be the number of  Jordan blocks of size 1 (recall that the condition on the rank  requires that $ 2(s+1)\le n$). The construction of $N$ also depends on a certain parameter, called $r$, that must satisfy $ 1+2s+r<n$ and that it is zero in the above examples. With such ingredients, we are
going to define a family of zero-square matrices called $N_{s,r}$ and an auxiliary family of zero-square matrices denoted $N'_{s,r}$.

\begin{definition}\label{nilp}
Let $s,r\in \mathbb{N}\cup\{0\}$ such that $ 2(s+1)\le n$ and $ 1+2s+r<n$. Let $A\in \mathbb{M}_n(\FF)$ be a nilpotent matrix consisting on a single Jordan block and $s$ Jordan blocks of size $1$  (notice that the condition $ 2(s+1)\le n$ is equivalent to $A$ having rank greater or equal to $\frac n2$). Let us define the matrices
\begin{align*}
   N_{0,r}:&=-e_{1,n},\\
N_{s,r}:&=\sum_{i=n-2s-r}^{n-s-r-2} (e_{i+1,i}+e_{i+1+s+r+1,i})+e_{1,n-s-r-1}\\
&   -\sum_{i=n-s}^{n-1} (e_{i+1,i}+e_{i+1-s-r-1,i})-e_{1,n},\hbox{ if $s\ge 1$, }\\
N'_{0,r}:&=-e_{1,n},\\
N'_{s,r}:&=Q_{n-2s-r,n-s-r-1}N_{s,r}Q_{n-2s-r,n-s-r-1}=\\
&=\sum_{i=n-2s-r}^{n-s-r-2} (e_{i+1,i}-e_{i+1+s+r+1,i})-e_{1,n-s-r-1}\\
&   -\sum_{i=n-s}^{n-1} (e_{i+1,i}-e_{i+1-s-r-1,i})-e_{1,n},\hbox{  if $s\ge 1$,  }\\
T_{s,r}:&=A-N_{s,r},\hbox{  for  $s\ge 0$,}\\
T'_{s,r}:&=A-N'_{s,r}\hbox{  for  $s\ge 0$.}
\end{align*}
\end{definition}

Now we are going to show that the  matrices $N_{s,r}$ and $N'_{s,r}$ are zero-square matrices. Let us recall the elementary matrices of $\mathbb{M}_n(\FF)$:
\begin{itemize}
  \item[] $P_{i,j}(t):=\Id+t e_{i,j}$, $t\in \FF$.
  \item[] $P_{i,j}:=\Id+ e_{i,j}+e_{j,i}-e_{i,i}-e_{j,j}$
  \item[] $Q_i(s)=\Id+(s-1)e_{i,i}$, $0\neq s\in \FF$.
\end{itemize}
All of them are invertible with inverses $(P_{i,j}(t))^{-1}=P_{i,j}(-t)$, $(P_{i,j})^{-1}=P_{i,j}$ and $(Q_i(s))^{-1}=Q_i(1/s)$.
Let us define $Q_{i,j}:=\sum_{r=i}^j Q_r(-1)$.

\begin{proposition}\label{nilp2}
Let $\FF$ be a field and take $n\in \mathbb{N}$ and $s,r\in \mathbb{N}\cup\{0\}$ such that $ 2(s+1)\le n$ and $ 1+2s+r<n$. Then, the matrices $N_{s,r}$ and $N'_{s,r}$ in $\mathbb{M}_n(\FF)$ are nilpotent of index $2$.
\end{proposition}

\begin{proof} Along this proof, let us identify matrices with the endomorphisms acting on the canonical basis $\{e_1,\dots,e_n\}$ of $\FF^n$. Clearly $N_{0,r}^2=(N'_{0,r})^2=0$. Suppose that $s\ge 1$ and let us denoted by $N$ the matrix $N_{s,r}$. For every $i=1\dots, n$ we have:

\begin{itemize}
\item[(1)] if $ i=1,\dots, n-2s-r-1$, $N(e_i)=0$,
\item[(2)] if $ i=n-2s-r,\dots, n-s-r-2$, $N(e_{i})=e_{i+1}+e_{i+1+s+r+1}$,
\item[(3)]  $N(e_{n-s-r-1})=e_1$,
\item[(4)] if $i=n-s-r,\dots, n-s-1$, $N(e_{i})=0$,
\item[(5)] if $i=n-s,\dots, n-1$, $N(e_{i})=-e_{i+1}-e_{i+1-s-r-1}$,
\item[(6)] $N(e_n)=-e_1$
\end{itemize}

Therefore, the image of $N$ is spanned by

$$
\{e_1\}\cup\{e_{n-2s-r+1}+e_{n-s+2}, \dots, e_{n-s-r-1}+e_{n}\}\cup  \{e_{n-s+1}+e_{n-2s-r}, \dots, e_{n}+e_{n-s-r-1}\},
$$ i.e., the image of $N$ is  the subspace generated by
$$S=\{e_1\}\cup \{e_{n-s+i}+e_{n-2s-r-1+i}\quad |\quad i=1,\dots, s\}.$$
Let us see that $N(S)=0$. Clearly, $N(e_1)=0$; moreover, if $i=1,\dots, s-1$, we have

\begin{align*}
      & N(e_{n-s+i}+e_{n-2s-r-1+i})=\\
      &=-e_{n-s+i+1}-e_{n-2s-r+i}+e_{n-2s-r+i}+e_{n-s+i+1}=0
\end{align*}
and $N(e_{n}+e_{n-s-r-1})=-e_1+e_1=0$. Therefore, $N^2=0$.

Since the matrix $Q_{n-2s-r,n-s-r-1}$ satisfies $Q_{n-2s-r,n-s-r-1}^2=\Id$, we deduce that $N'_{s,r}$ is also a zero-square matrix, as required.
\end{proof}

In Theorem \ref{charpoly} we will calculate the characteristic polynomials of the matrices
$$
T_{s,r}:=A-N_{s,r}\quad \hbox{ and }\quad T'_{s,r}:=A-N'_{s,r},
$$
defined in \ref{nilp}.
Since characteristic polynomials are invariant under similarity, our argument will consist on transforming by similarity the matrices $T_{s,r}$ and $T'_{s,r}$ in a finite number of steps into matrices of the form $T_{s',r'}$ or $T'_{s',r'}$, $s'=0,1$ or 2.
We will first prove two descending lemmas.
In each of these descending lemmas,   we will decrease the first subscript of $T_{s,r}$ or of $T'_{s,r}$ by three units and increase the second subscript by three units. Afterwards, it will only remain to explicitly calculate the characteristic polynomial of the matrices $T_{s',r'}$ and $T'_{s',r'}$, $s'=0,1,2$.

\begin{lemma}[From $T_{s,r}$ to $T'_{s-3,r+3}$]\label{fromTtoT'}
Let $\FF$ be a field and let $A\in \mathbb{M}_n(\FF)$ be a nilpotent matrix consisting on a single Jordan block of rank bigger than or equal to $\frac n2$ and $s$ Jordan blocks of size $1$. Suppose that $s\ge 3$ and let $r\in \mathbb{N}\cup \{0\}$ such that $ 1+2s+r<n$. Let $T_{s,r}$ and $T'_{s,r}$ be the matrices defined in \ref{nilp}.
Then, the matrix $T_{s,r}$ is similar to the matrix $T'_{s-3,r+3}$.
\end{lemma}

\begin{proof}  For $v=s+r+1$ and
\begin{align*}
  T_1:&=P_{n-s+3,n-s+3-v}(-1)\cdot T_{s,r}\cdot P_{n-s+3,n-s+3-v}(1)\\
   T_2:&=P_{n-s+2,n-s+2-v}\cdot T_1\cdot P_{n-s+2,n-s+2-v}\\
   T_3:&=P_{n-s+1-v,n-s+1}(-1)\cdot T_2\cdot P_{n-s+1-v,n-s+1}(1),\\
   T_4:&=Q_{n-s+2-v,n-s+2}\cdot T_3\cdot Q_{n-s+2-v,n-s+2},
\end{align*}
we will show that $T_4=T'_{s-3,r+3}$.

Along this proof, let us identify matrices with the endomorphisms acting on the canonical basis $\{e_1,\dots,e_n\}$ of $\FF^n$.
Let us denote by $P$ the following matrix $$P:=P_{n-s+3,n-s+3-v}(1)\cdot P_{n-s+2,n-s+2-v}\cdot P_{n-s+1-v,n-s+1}(1)\cdot Q_{n-s+2-v,n+r+3-v}.$$
With this notation in hand, we have $T_4=P^{-1}T_{s,r}P$. Let us calculate $P(e_i)$, $i=1,\dots,n$:
\begin{itemize}
\item[] $P(e_i)=e_i$  for $i=1,2,\dots, n-s-v$,
\item[] $P(e_{n-s-v+1})=e_{n-s-v+1}$,
\item[] $P(e_{n-s-v+2})=-e_{n-s+2}$,
\item[] $P(e_{n-s-v+3})=-e_{n-s-v+3}-e_{n-s+3}$,
\item[] $P(e_i)=-e_i$ for  $i=n-s-v+4, \dots, n-v$ \quad  (if $n-s-v+4\le n-v$),
\item[] $ P(e_i)=-e_i$ for $i= n-v+1,\dots, n-s-1$ \quad  (if $n-v+1\le n-s-1$),
\item[] $P(e_{n-s})=-e_{n-s}$ \quad  (if $s\ne 3$),
\item[] $P(e_{n-s+1})=-e_{n-s+1}-e_{n-s+1-v}$,
\item[] $P(e_{n-s+2})=-e_{n-s+2-v}$ and
\item[] $P(e_i)=e_i$ for  $i=n-s+3,\dots, n$.
\end{itemize}

Let us see how $T_4=P^{-1}T_{s,r}P$ acts on the canonical basis: to do so let us divide the set $\{1,2,\dots, n\}$ into the following subsets relative to the definition of $T'_{s-3,r+3}$:
$$
\begin{array}{ll}
\Delta'_1:=\{1,\dots,n-s-v+3\} & \Delta'_2:=\{n-s-v+4,\dots, n-v-1\}\\
\Delta'_3:=\{n-v, \text{ if } s\ne 3\} &
  \Delta'_4:=\{n-v+1,\dots, n-s+2\}\\
  \Delta'_5:=\{n-s+3,\dots,n-1\}&\Delta'_6:=\{n\}
\end{array}
$$
(notice that if $s=3$, then $\Delta'_2=\emptyset$, $\Delta'_3=\emptyset$ and $\Delta'_5=\emptyset$; and if $s= 4$, then $\Delta'_2=\emptyset$).

\noindent $\bullet$ $T_4$ acting on elements of $\Delta'_1$:

\begin{itemize}
\item[(1)] $T_4(e_i)=P^{-1}T_{s,r}P(e_i)=P^{-1}T_{s,r}(e_i)=P^{-1}(e_{i+1})=e_{i+1}$  for $ i=1,2,\dots, n-s-v$,
\item[(2)] $T_4(e_{n-s-v+1})=P^{-1}T_{s,r}P(e_{n-s-v+1})=P^{-1}T_{s,r}(e_{n-s-v+1})=P^{-1}(-e_{n-s+2})=e_{n-s-v+2}$,
\item[(3)] $T_4(e_{n-s-v+2})=P^{-1}T_{s,r}P(e_{n-s-v+2})=P^{-1}T_{s,r}(-e_{n-s+2})=P^{-1}(-e_{n-s+3}-e_{n-s-v+3})=e_{n-s-v+3}$,
\item[(4)] $T_4(e_{n-s-v+3})=P^{-1}T_{s,r}P(e_{n-s-v+3})=P^{-1}T_{s,r}(-e_{n-s+3}-e_{n-s-v+3})=P^{-1}(-e_{n-s+4}-e_{n-s-v+4}+e_{n-s+4})=P^{-1}(-e_{n-s-v+4})=e_{n-s-v+4}$.
\end{itemize}

\noindent $\bullet$ $T_4$ acting on elements of $\Delta'_2$ (when $\Delta_2'\ne \emptyset$):

\begin{itemize}
\item[(5)] $T_4(e_i)=P^{-1}T_{s,r}P(e_i)=P^{-1}T_{s,r}(-e_i)=P^{-1}(e_{i+1+v})=e_{i+1+v}$ for $i=n-s-v+4, \dots, n-v-1$.
\end{itemize}

\noindent $\bullet$ $T_4$ acting on elements of $\Delta'_3$ (when $\Delta_3'\ne \emptyset$):

\begin{itemize}
\item[(6)] $T_4(e_{n-v})=P^{-1}T_{s,r}P(e_{n-v})=P^{-1}T_{s,r}(-e_{n-v})=P^{-1}(e_1-e_{n-v+1})=e_1+e_{n-v+1}$.
\end{itemize}

\noindent $\bullet$ $T_4$ acting on elements of $\Delta'_4$:
\begin{itemize}
\item[(7)] $T_4(e_i)=P^{-1}T_{s,r}P(e_i)=P^{-1}T_{s,r}(-e_i)=P^{-1}(-e_{i+1})=e_{i+1}$ for $i=n-v+1, \dots, n-s-1$,
\item[(8)] $T_4(e_{n-s})=P^{-1}T_{s,r}P(e_{n-s})=P^{-1}T_{s,r}(-e_{n-s})=P^{-1}(-e_{n-s+1}-e_{n-s+1-v})\\ =e_{n-s+1}$,
\item[(9)] $T_4(e_{n-s+1})=P^{-1}T_{s,r}P(e_{n-s+1})=P^{-1}T_{s,r}(-e_{n-s+1}-e_{n-s+1-v}) =\\ P^{-1}(-e_{n-s+2}-e_{n-s+2-v}+e_{n-s+2})=
P^{-1}(-e_{n-s+2-v})=e_{n-s+2}$,
\item[(10)] $T_4(e_{n-s+2})=P^{-1}T_{s,r}P(e_{n-s+2})=P^{-1}T_{s,r}(-e_{n-s+2-v})=P^{-1}(e_{n-s+3})=e_{n-s+3}$.
\end{itemize}

\noindent $\bullet$ $T_4$ acting on elements of $\Delta'_5$ (when $\Delta_5'\ne \emptyset$):
\begin{itemize}
\item[(11)] $T_4(e_{n-s+3})=P^{-1}T_{s,r}P(e_{n-s+3})=P^{-1}T_{s,r}(e_{n-s+3})=P^{-1}(e_{n-s+4}+e_{n-s+4-v})=e_{n-s+4}-e_{n-s+4-v}$,
\item[(12)] $T_4(e_i)=P^{-1}T_{s,r}P(e_i)=P^{-1}T_{s,r}(e_i)=P^{-1}(e_{i+1}+e_{i+1-v})=e'_{i+1}-e'_{i+1-v}$ for $i=n-s+4, \dots, n-1$.
\end{itemize}

\noindent $\bullet$ $T_4$ acting on elements of $\Delta'_6$:
\begin{itemize}
\item[(13)] $T_4(e_n)=P^{-1}T_{s,r}P(e_n)=P^{-1}T_{s,r}(e_n)=P^{-1}(e_{1})=e_1$,
\end{itemize}
which proves that $T_4=P^{-1}T_{s,r}P$ behaves on the canonical basis in the same way as $T'_{s-3,r+3}$ so $T_4=T'_{s-3,r+3}$.
\end{proof}

Following the same arguments as in the above proof we also obtain the following descending lemma.

\begin{lemma}[From $T'_{s,r}$ to $T_{s-3,r+3}$]\label{fromT'toT}
Let $\FF$ be a field and let $A\in \mathbb{M}_n(\FF)$ be a nilpotent matrix consisting on a single Jordan block of rank bigger than or equal to $\frac n2$ and $s$ Jordan blocks of size $1$. Suppose that $s\ge 3$ and let $r\in \mathbb{N}\cup \{0\}$ such that $ 1+2s+r<n$.  Let $T_{s,r}$ and $T'_{s,r}$ be the matrices defined in \ref{nilp}.
Then the matrix $T'_{s,r}$ is similar to the matrix $T_{s-3,r+3}$. Indeed, if $v=s+r+1$ and
\begin{align*}
  T'_1:&=P_{n-s+3,n-s+3-v}(1)\cdot T'_{s,r}\cdot P_{n-s+3,n-s+3-v}(-1)\\
   T'_2:&=P_{n-s+2,n-s+2-v}\cdot T'_1\cdot P_{n-s+2,n-s+2-v}\\
   T'_3:&=P_{n-s+1-v,n-s+1}(1)\cdot T'_2\cdot P_{n-s+1-v,n-s+1}(-1),   \\
   T'_4:&=Q_{n-v-s+3,n-s+1}\cdot T'_3\cdot Q_{n-v-s+3,n-s+1},
\end{align*}
we obtain that $T'_4=T_{s-3,r+3}$.
\end{lemma}

After these two descending lemmas, let us compute the characteristic polynomials of $T_{s,r}$ and $T'_{s,r}$ when $s=0,1$ or $2$.

\begin{remark}[Case $s=0$]\label{case0}
Let $\FF$ be a field and let $A\in \mathbb{M}_n(\FF)$ be a nilpotent matrix consisting on a single Jordan block of rank bigger than or equal to $\frac n2$ and no Jordan blocks of size $1$ ($s=0$). For any $r<n-1$ we have that $T_{0,r}=T'_{0,r}=\sum_{t=1}^{n-1}e_{t+1,t}+e_{1,n}$, and its characteristic polynomial is exactly $x^n-1$.
\end{remark}

\begin{lemma} [Case $s=1$] \label{case1}
Let $\FF$ be a field and let $A\in \mathbb{M}_n(\FF)$ be a nilpotent matrix consisting on a single Jordan block of rank bigger than or equal to $\frac n2$ and one Jordan block of size $1$. Let $r\ge 0$ be such that $r+3<n$. Let us consider the matrices $T_{1,r}=A-N_{1,r}$ and $T'_{1,r}=A-N'_{1,r}$.
Then, the characteristic polynomials are
\begin{itemize}
\item $p_{T_{1,r}}(x)=x^n+x^{r+2}-x^{n-r-2}-1=(x^{r+2}-1)(x^{n-r-2}+1)$
\item $p_{T'_{1,r}}(x)=x^n-x^{r+2}+x^{n-r-2}-1=(x^{r+2}+1)(x^{n-r-2}-1)$.
\end{itemize}
\end{lemma}

\begin{proof} We will use the following general remark: suppose we are computing the determinant of a matrix of the form
\begin{align*}
B=\left( \begin{array}{ccccc}
             x & 0 &  \cdots &0& 1 \\
             -1 & x &  \cdots & & 0 \\
             \vdots  & \ddots & \ddots & \vdots&\vdots  \\
              0 & 0  &\ddots & x & 0 \\
             0 & 0 &  \cdots & -1& x \\
           \end{array}
         \right)+a\, e_{i,j}
\end{align*} where $a$ is a non-zero element at position $(i,j)$, $j>i$. If we multiply $B$ on the left by $P_{i,j+1}(a)$ (we add to row $i$ the row $j+1$ multiplied by $a$), we get a new matrix that has a zero at position $(i,j)$ and has $ax$ at position $(i,j+1)$. Since $|P_{i,j+1}(a)|=1$ we get
\begin{align*}
&\det\left(\left( \begin{array}{ccccc}
             x & 0 &  \cdots &0& 1 \\
             -1 & x &  \cdots & & 0 \\
             \vdots  & \ddots & \ddots & \vdots&\vdots  \\
              0 & 0  &\ddots & x & 0 \\
             0 & 0 &  \cdots & -1& x \\
           \end{array}
         \right)+a\, e_{i,j}\right)=\\
      &   \det\left(\left( \begin{array}{ccccc}
             x & 0 &  \cdots &0& 1 \\
             -1 & x &  \cdots & & 0 \\
             \vdots  & \ddots & \ddots & \vdots&\vdots  \\
              0 & 0  &\ddots & x & 0 \\
             0 & 0 &  \cdots & -1& x \\
           \end{array}
         \right)+ax\, e_{i,j+1}\right)
          \end{align*}
i.e., when computing the determinant of $B$, nonzero elements at positions $(i,j)$, $j>i$, can be moved to the right by multiplying  by $x$ at each step.

By a similar argument, multiplying $B$ on the right by the elementary matrix $P_{i-1,j}(a)$ (adding to column $j$ the column $i-1$ multiplied by $a$), we also have that
\begin{align*}
&\det\left(\left( \begin{array}{ccccc}
             x & 0 &  \cdots &0& 1 \\
             -1 & x &  \cdots & & 0 \\
             \vdots  & \ddots & \ddots & \vdots&\vdots  \\
              0 & 0  &\ddots & x & 0 \\
             0 & 0 &  \cdots & -1& x \\
           \end{array}
         \right)+a\, e_{i,j}\right)=\\
      &   \det\left(\left( \begin{array}{ccccc}
             x & 0 &  \cdots &0& 1 \\
             -1 & x &  \cdots & & 0 \\
             \vdots  & \ddots & \ddots & \vdots&\vdots  \\
              0 & 0  &\ddots & x & 0 \\
             0 & 0 &  \cdots & -1& x \\
           \end{array}
         \right)+ax\, e_{i-1,j}\right)
          \end{align*}
i.e., when computing the determinant of $B$, nonzero elements at positions $(i,j)$, $j>i$, can be moved upwards to position $(i-1,j)$ by multiplying at each step by $x$.

Notice that
\begin{align*}
     T_{1,r}&  =\sum_{t=1}^{n-1}e_{t+1,t}+e_{1,n}+e_{n-r-2,n-1}-e_{1,n-r-2}.\\
\end{align*}

Applying the argument above to the determinant of the matrix $x\Id -T_{1,r}$, we can move the nonzero element at position $(n-r-2,n-1)$ towards position $(1,n)$ obtaining $-x^{n-r-2}$ and we can move the non-zero element at position $({1,n-r-2})$ towards position $(1,n)$ obtaining $x^{r+2}$, getting at the end the determinant of a matrix of the form

\bigskip

\begin{align*}
\left( \begin{array}{cccccc}
             x & 0 & 0 & \cdots &0& -x^{n-r-2}+x^{r+2}-1 \\
             -1 & x & 0 & \cdots && 0 \\
             0 & -1 & x & \cdots &0& 0 \\
             \vdots & \vdots & \ddots & \ddots & \vdots&\vdots  \\
              0 & 0 & 0 &\ddots & x & 0 \\
             0 & 0 & 0 & \cdots & -1& x \\
           \end{array}
         \right).
\end{align*}

\bigskip

Finally, by the Laplace expansion of the determinant along the first row we get that the characteristic polynomial of $ T_{1,r}$ is given by $x^n-x^{n-r-2}+x^{r+2}-1$.

Similarly, since $$T'_{1,r} =\sum_{t=1}^{n-1}e_{t+1,t}+e_{1,n}-e_{n-r-2,n-1}+e_{1,n-r-2},$$ its characteristic polynomial is $x^n+x^{n-r-2}-x^{r+2}-1$.
\end{proof}

We continue our work with the next technical assertion.

\begin{lemma} [Case $s=2$] \label{case2}
Let $\FF$ be a field and let $A\in \mathbb{M}_n(\FF)$ be a nilpotent matrix consisting on a single Jordan block of rank bigger than or equal to $\frac n2$ and two Jordan blocks of size $1$. Let $r\ge 0$ be such that $r+5<n$. Let  us consider the matrices $T_{2,r}=A-N_{2,r}$ and $T'_{2,r}=A-N'_{2,r}$.
Then, the characteristic polynomials are
\begin{itemize}
\item $p_{T_{2,r}}(x)=x^n+x^{r+3}-x^{n-r-3}-1=(x^{n-r-3}+1)(x^{r+3}-1)$
\item $p_{T'_{2,r}}(x)=x^n-x^{r+3}+x^{n-r-3}-1=(x^{n-r-3}-1)(x^{r+3}+1)$.
\end{itemize}
\end{lemma}

\begin{proof} Let $\{e_1,\dots, e_n\}$ be the canonical basis of $\mathbb{F}^n$.
Let
$$
P:= P_{n,n-r-3}\cdot Q_{n-r-3,n-r-3}
$$
and consider the new basis $$
\{e'_{i}, |\ i=1,\dots, n\}, \hbox{  where each $e'_{i}:=P(e_i)$.}
$$
This new basis is just a reordering of the canonical basis and a change of sign of one of its vectors:
$$
 \{e'_{i}, |\ i=1,\dots, n\}= \{e_1,e_2,\dots, e_{n-r-4},-e_n\}\cup \{e_{n-r-3},\dots, e_{n-1}\}.
$$

Recall that
\begin{align*}
 &T_{2,r}  =\sum_{t=1}^{n-r-5} e_{t+1,t}-e_{n,n-r-4}   +e_{1,n}      \\
    & +\sum_{t=n-r-3}^{n-1} e_{t+1,t}-e_{1,n-r-3}  + e_{n-r-4,n-2}+ e_{n-r-1,n-1}.  \\
\end{align*}

Then, the matrix $P^{-1}\cdot T_{2,r}\cdot P$ is of the form

\bigskip

$$
\left(
  \begin{array}{cccc|cccc}
    0 & 0     & 0 & -1 & * & *& * & * \\
    1 & 0     & 0 & 0 & * & * & * & *\\
    0 & \ddots & 0 & 0 & * & * & * & *\\
    0 & 0     & 1 & 0 & * & * & * & *\\
    \hline
    0 & 0     & 0 & 0 & 0 & 0 & 0 & 1 \\
    0 & 0     & 0 & 0 & 1 & 0 & 0 & 0\\
    0 & 0     & 0 & 0 & 0 & \ddots& 0 & 0 \\
    0 & 0     & 0 & 0 & 0 & 0& 1 & 0 \\
  \end{array}
\right)
$$

\bigskip

\noindent and its characteristic polynomial, which coincides with the characteristic polynomial of $T_{2,r}$, is the product of the characteristic polynomials of the two diagonal blocks, i.e., it is precisely $(x^{n-r-3}+1)(x^{r+3}-1).$

Similarly, if we consider $P':= P_{n,n-r-3}\cdot Q_{n,n}$ and the new basis
$$
\{e''_{i}, |\ i=1,\dots, n\},\ \hbox{where each $e''_{i}:=P'(e_i)$,}
$$  then,
$$
 \{e''_{i}, |\ i=1,\dots, n\}= \{e_1,e_2,\dots, e_{n-r-4},e_n\}\cup \{e_{n-r-3},\dots, -e_{n-1}\}.
$$
Thus, the matrix $(P')^{-1}\cdot T'_{2,r}\cdot P'$ is of the form

\medskip

$$
\left(
  \begin{array}{cccc|cccc}
    0 & 0     & 0 & 1 & * & *& * & * \\
    1 & 0     & 0 & 0 & * & * & * & *\\
    0 & \ddots & 0 & 0 & * & * & * & *\\
    0 & 0     & 1 & 0 & * & * & * & *\\
    \hline
    0 & 0     & 0 & 0 & 0 & 0 & 0 & -1 \\
    0 & 0     & 0 & 0 & 1 & 0 & 0 & 0\\
    0 & 0     & 0 & 0 & 0 & \ddots& 0 & 0 \\
    0 & 0     & 0 & 0 & 0 & 0& 1 & 0 \\
\end{array}
\right)
$$

\medskip

\noindent and its characteristic polynomial, which coincides with the characteristic polynomial of $T'_{2,r}$, is the product of the characteristic polynomials of the two diagonal blocks, i.e., it is precisely
$(x^{n-r-3}-1)(x^{r+3}+1).$
\end{proof}

We thus arrive at the following result in which we compute the characteristic polynomials of the matrices $T_{s,r}=A-N_{s,r}$ and $T'_{s,r}=A-N'_{s,r}$. They depend on the equivalence class of $s$ modulo 3.

\begin{theorem}\label{charpoly}
Let $\FF$ be a field,  $n\in \mathbb{N}$, and let $s\in \mathbb{N}\cup\{0\}$ such that $ 2(s+1)\le n$. Let $A\in \mathbb{M}_n(\FF)$ be a nilpotent matrix consisting on a single Jordan block and $s$ Jordan blocks of size $1$.
Then,
\begin{enumerate}
\item[(1)] If  $s=3\alpha$ for some $\alpha\ge 0$ and  we take any $r$ such that $1+2s+r<n$, then the characteristic polynomial of $T_{s,r}=A-N_{s,r}$ is $$p_{T_{3\alpha,r}}(x)=x^n-1.$$

\item[(2.1)] If $s=3\alpha+1$ for some even $\alpha\ge 0$ and we take any $r$ such that $1+2s+r<n$  then  the characteristic polynomial of $T_{s,r}=A-N_{s,r}$
is given by
$$
   p_{T_{3\alpha+1,r}}(x)=p_{T_{1,r+3\alpha}}(x)= (x^{r+3\alpha+2}-1)(x^{n-r-3\alpha-2}+1),
    $$
and the characteristic polynomial of $T'_{s,r}=A-N'_{s,r}$
is given by
$$
   p_{T'_{3\alpha+1,r}}(x)=p_{T'_{1,r+3\alpha}}(x)= (x^{r+3\alpha+2}+1)(x^{n-r-3\alpha-2}-1).
    $$
 \item[(2.2)]  If $s=3\alpha+1$ for some odd $\alpha\ge 1$,   and we take any $r$ such that $1+2s+r<n$  then  the characteristic polynomial of $T_{s,r}=A-N_{s,r}$
is given by
$$
   p_{T_{3\alpha+1,r}}(x)=p_{T'_{1,r+3\alpha}}(x)= (x^{r+3\alpha+2}+1)(x^{n-r-3\alpha-2}-1),
    $$
and the characteristic polynomial of $T'_{s,r}=A-N'_{s,r}$
is given by
$$
   p_{T'_{3\alpha+1,r}}(x)=p_{T_{1,r+3\alpha}}(x) =(x^{r+3\alpha+2}-1)(x^{n-r-3\alpha-2}+1).
    $$
  \item[(3.1)] If $s=3\alpha+2$ for some even $\alpha\ge 0$    and we take any $r$ such that $1+2s+r<n$, then  the characteristic polynomial of $T_{s,r}=A-N_{s,r}$
is given by
$$
   p_{T_{3\alpha+2,r}}(x)=p_{T_{2,r+3\alpha}}(x)=(x^{n-r-3\alpha-3}+1)(x^{r+3\alpha+3}-1),
$$
and the characteristic polynomial of $T'_{s,r}=A-N'_{s,r}$
is given by
$$
   P_{T'_{3\alpha+2,r}}(x)=p_{T'_{2,r+3\alpha}}(x) =(x^{n-r-3\alpha-3}-1)(x^{r+3\alpha+3}+1).
    $$
  \item[(3.2)]   If $s=3\alpha+2$   for some odd $\alpha\ge 0$ and we take any $r$ such that $1+2s+r<n$, then  the characteristic polynomial of $T_{s,r}=A-N_{s,r}$
is given by
$$
   p_{T_{3\alpha+2,r}}(x)=p_{T'_{2,r+3\alpha}}(x)=(x^{n-r-3\alpha-3}-1)(x^{r+3\alpha+3}+1),
$$
and the characteristic polynomial of $T'_{s,r}=A-N'_{s,r}$
is given by
$$
   p_{T'_{3\alpha+2,r}}(x)=p_{T_{2,r+3\alpha}}(x) =(x^{n-r-3\alpha-3}+1)(x^{r+3\alpha+3}-1).
$$

\end{enumerate}
\end{theorem}

\begin{proof} Notice that the condition $ 2(s+1)\le n$ is just equivalent to $A$ having rank bigger or equal to $\frac n2$.

\noindent (1) Suppose that $s=3\alpha$ for some $\alpha\ge 0$.  Let us consider $T_{3\alpha,r}=A-N_{3\alpha,r}$. If $\alpha\ge 1$, applying Lemma \ref{fromTtoT'} to $T_{3\alpha,r}$ we obtain the  matrix $T'_{3(\alpha-1),r+3}$, which is similar to $T_{3\alpha,r}$. If $\alpha-1\ge 1$ then use Lemma \ref{fromT'toT} to get another similar matrix $T_{3(\alpha-2),r+6}$, and repeating this process, in $\alpha$ steps we will end with either $ T_{0,r+3\alpha}$ or with $T'_{0,r+3\alpha}$, which are both equal to the matrix $\sum_{t=1}^{n-1}e_{t+1,t}+e_{1,n}$ (see Remark \ref{case0}), whose characteristic polynomial is $x^n-1$. Thus the characteristic polynomial of $T_{s,r}$ is  $x^n-1$.

\noindent (2.1) Suppose that $s=3\alpha+1$ for some even $\alpha\ge 0$. Let us consider $T_{s,r}$.  If $\alpha>1$, as in (1) we can use Lemma \ref{fromTtoT'}  and Lemma \ref{fromT'toT} several times  to obtain, in $\alpha$ steps, a similar matrix $ T_{1,r+3\alpha}$; therefore, the  characteristic polynomial of $T_{s,r}$ coincides with the characteristic polynomial of $T_{1,r+3\alpha}$, which is $(x^{r+3\alpha+2}-1)(x^{n-r-3\alpha-2}+1)$ by Lemma \ref{case1}. Similarly, if we begin with $T'_{s,r}$ and we use Lemmas \ref{fromT'toT} and \ref{fromTtoT'}, we will end up in $\alpha$ steps with the similar matrix $T'_{1,r+3\alpha}$, whose characteristic polynomial is $(x^{r+3\alpha+2}+1)(x^{n-r-3\alpha-2}-1)$ by Lemma \ref{case1}.

\noindent (2.2) Suppose that $s=3\alpha+1$ for some odd $\alpha\ge 1$. If we start with $T_{s,r}$ and apply Lemmas \ref{fromTtoT'} and  \ref{fromT'toT}, after $\alpha$ steps we will end up with the similar matrix $T'_{1,r+3\alpha}$, and if we begin with $T'_{s,r}$ we will end up with $T_{1,r+3\alpha}$.

\noindent (3.1) and (3.2) follow as (1) and (2.1) and (2.2) above, taking into account the characteristic polynomials of $T_{2,r+3\alpha}$ and $T'_{2,r+3\alpha}$ calculated in Lemma \ref{case2}.
\end{proof}

As three important consequences, we derive the following:

\begin{corollary}\label{evenn}
Let $\FF$ be a field, suppose that $n\in \mathbb{N}$ is even, and let $s\in \mathbb{N}\cup\{0\}$ such that $ 2(s+1)\le n$. Let $A\in \mathbb{M}_n(\FF)$ be a nilpotent matrix consisting on a single Jordan block and $s$ Jordan blocks of size $1$. If we take $r=\frac n2-s-1$, the characteristic polynomial of $T_{s,r}=A-N_{s,r}$ and the characteristic polynomial of $T'_{s,r}=A-N_{s,r}$ coincide and they are $x^n-1$. In particular, $T_{s,r}^n=\Id$ so that $A$ decomposes as the sum of a torsion matrix $T_{s,r}$ and a zero-square matrix $N_{s,r}$.
\end{corollary}

\begin{proof}
If $s\equiv_30$, since the characteristic polynomial of $T_{s,r}$ is always $x^n-1$ for any $r$ such that $1+2s+r<n$, this also is true if we take $r=\frac n2-s-1$.

If $s\equiv_31$ or $s\equiv_32$ and we take $r=\frac n2-s-1$, we obtain that the characteristic polynomial of $T_{s,r}$ is $(x^{\frac n2}-1)(x^{\frac n2}+1)=x^n-1$.
\end{proof}

\begin{corollary}
Let $\FF$ be a field, suppose that $n\in \mathbb{N}$ is odd, and let $s\in \mathbb{N}\cup\{0\}$ such that $ 2(s+1)\le n$. Let $A\in \mathbb{M}_n(\FF)$ be a nilpotent matrix consisting on a single Jordan block and $s$ Jordan blocks of size $1$.
\begin{enumerate}
\item If $s\equiv_30$ and we take any $r$ such that $1+2s+r<n$, then $N_{s,r}$ such that the characteristic polynomial of $T_{s,r}=A-N_{s,r}$ is $x^n-1$.
\item If $s\equiv_31$ or $s\equiv_32$, and:
\begin{enumerate}
\item if $n=4m+1$, $\alpha$ is even, $s=3\alpha+1$ or $s=3\alpha+2$, and we take $r=2m-s$, then $N_{s,r}$ is such that the characteristic polynomial of $T_{s,r}=A-N_{s,r}$ is $(x^{2m}+1)(x^{2m+1}-1)$;
\item if $n=4m+1$, $\alpha$ is odd, $s=3\alpha+1$ or $s=3\alpha+2$, and we take $r=2m-s$, then $N'_{s,r}$ is such that the characteristic polynomial of $T'_{s,r}=A-N'_{s,r}$ is $(x^{2m}+1)(x^{2m+1}-1)$;
\item if $n=4m+3$, $\alpha$ is even, $s=3\alpha+1$ or $s=3\alpha+2$, and we take $r=2m-s$, then  $N_{s,r}$ is such that the characteristic polynomial of $T_{s,r}=A-N_{s,r}$ is $(x^{2m+2}+1)(x^{2m+1}-1)$;
\item if $n=4m+3$, $\alpha$ is odd, $s=3\alpha+1$ or $s=3\alpha+2$, and we take $r=2m-s$, then  $N'_{s,r}$ is such that the characteristic polynomial of $T'_{s,r}=A-N'_{s,r}$ is $(x^{2m+2}+1)(x^{2m+1}-1)$.
\end{enumerate}
\end{enumerate}
\end{corollary}

\begin{proof} It is an automatic consequence of Theorem \ref{charpoly}.
\end{proof}

\begin{corollary}\label{oddn}
Let $\FF$ be a field of zero characteristic, suppose that $n\in \mathbb{N}$ is odd, and let $s\in \mathbb{N}\cup\{0\}$ such that $2(s+1)\le n$. Let $A\in \mathbb{M}_n(\FF)$ be a nilpotent matrix consisting on a single Jordan block and $s$ Jordan blocks of size $1$. Then, there exists a zero-square matrix $N$ and $d\ge 1$ such that $(A-N)^d=\Id$, i.e., $A$ decomposes as the sum of a torsion matrix and a zero-square matrix.
\end{corollary}

\begin{proof} Following the cases of the above corollary, we have:
\begin{itemize}
  \item[-] if we are in case (1), take $N=N_{s,r}$ and $d=n$; clearly $(A-N)^n=\Id$;
  \item[-] if we are in case (2)(a), take $r=2m-s$ and $N=N_{s,r}$; then if $d=lcm(4m,2m+1)$, we obtain $(A-N)^d=\Id$, because the polynomials $x^{2m}+1$ and $x^{2m+1}-1$ have no common roots, so the matrix $A-N$ has $n$ different eigenvalues (in the algebraic closure of $\mathbb{F}$), all of them roots of unity $\xi$ with $\xi^{4m}=1$ or $\xi^{2m+1}=1$;
     \item[-] if we are in case (2)(b), take $r=2m-s$ and $N=N'_{s,r}$; then if $d=lcm(4m,2m+1)$, we obtain $(A-N)^d=\Id$ arguing as in case (2)(a);
     \item[-] if we are in case (2)(c), take $r=2m-s$ and $N=N_{s,r}$; then if $d=lcm(4m+2,2m+1)$, we obtain $(A-N)^d=\Id$ arguing as in case (2)(a);
     \item[-] if we are in case (2)(d), take $r=2m-s$ and $N=N'_{s,r}$; then if $d=lcm(4m+2,2m+1)$, we obtain $(A-N)^d=\Id$ arguing as in case (2)(a).
\end{itemize}
\end{proof}

We now have all the ingredients necessary to establish our chief result in this section.

\begin{theorem}\label{nilpotenttorsionmain}
Any nilpotent matrix in $\mathbb{M}_n(\FF)$ can be written as the sum of a torsion matrix plus a square-zero matrix if, and only if, its rank is at least $\frac n2$.
\end{theorem}

\begin{proof} According to the above   Theorem~\ref{nilpotentmain}, a necessary condition to express $A\in \mathbb{M}_n(\FF)$ as the sum $T+N$, where $T$ is a torsion matrix and $N^2=0$, is that the rank of $A$ is greater than or equal to $\frac n2$, because torsion matrices must be invertible. Conversely, to prove the sufficiency, suppose that the rank of $A\in \mathbb{M}_{n}(\mathbb{F})$ is no less than $\frac n2$.

If $\mathbb{F}$ is a field of prime characteristic, then the result follows by Corollary \ref{nilpotentprimefield} given above. Suppose from now on that $\mathbb{F}$ is a field of characteristic zero. By the primary rational canonical decomposition of any nilpotent matrix, we can assume without loss of generality that $A$ is the direct sum of its Jordan blocks.

If $A$ only contains Jordan blocks of size $r>1$, each of them can be decomposed separately in the following way

\bigskip

$$
\left(
          \begin{array}{cccc}
            0 & 0 & \dots  & 0 \\
            1 & 0 &  &  \vdots\\
             & \ddots & \ddots &  \\
            0 &  & 1 & 0\\
          \end{array}
        \right)=\left(
          \begin{array}{cccc}
            0 & 0 & \dots  & 1 \\
            1 & 0 &  &  \vdots\\
             & \ddots & \ddots &  \\
            0 &  & 1 & 0\\
          \end{array}
        \right)+\left(
          \begin{array}{cccc}
            0 &  \dots& \dots  & -1 \\
            \vdots &  &  &  \vdots\\
            \vdots &  &  &  \\
            0 &  \dots&\dots  & 0\\
          \end{array}
        \right)
$$

\bigskip

\noindent where the first matrix is a torsion matrix and the second one is zero-square.

Otherwise, since the rank of $A$ is greater than or equal to $\frac n2$, we can reorder the blocks of matrix $A$ -- which just corresponds to a reordering of the elementary divisors of $A$ -- as follows: we follow each  Jordan block of size $r>1$ by $s$ ($s\le r-2$) blocks of size $1$. If $r+s$ is even, the combined block of size $r+s$ can be decomposed as the sum of a torsion matrix and a zero-square matrix by virtue of Corollary \ref{evenn}. If, however, $r+s$ is odd, we elementarily see that the combined block of size $r+s$ can be decomposed as the sum of a torsion matrix and a zero-square matrix in view of Corollary \ref{oddn}.
\end{proof}

The next commentaries are worthwhile to show that the above theorem is somewhat true in more general settings.

\begin{remark} It is well known (see, for example, \cite{Co}, \cite{GGVdSMA1} or \cite{GGVdSMA2}) that every nilpotent matrix over a division ring is still similar to its Jordan form, i.e., it is similar to a direct sum of Jordan blocks, all of them associated to 0. Thus,
Theorem~\ref{nilpotenttorsionmain} is still valid for nilpotent matrices over division rings.
\end{remark}

The above result suggests the following question:

\medskip

\noindent{\bf Open Question 2:} Given a fixed index of nilpotence $k\ge 2$, is it true that any nilpotent matrix   in $\mathbb{M}_n(\FF)$ can be written as the sum of a torsion matrix and a nilpotent matrix of index less than or equal to $k$ if, and only if, its rank is at least $\frac nk$?

\bigskip

\noindent{\bf Funding:} The first-named author (Peter V. Danchev) was supported in part by the Bulgarian National Science Fund under Grant KP-06 No. 32/1 of December 07, 2019, the second-name author (Esther Garc\'{\i}a) was partially supported by Ayuda Puente 2022, URJC. The three authors were partially supported by the Junta de Andaluc\'{\i}a FQM264.

\vskip3.0pc

\bibliographystyle{plain}

\begin{thebibliography}{1}

\bibitem{AM}
A.~N. Abyzov and I.~I. Mukhametgaliev.
\newblock On some matrix analogues of the little Fermat theorem.
\newblock {\em Mat. Zametki}, {\bf 101}(2):163--168, 2017.

\bibitem{AT1}
A.~N. Abyzov and D.~T. Tapkin.
\newblock Rings over which every matrices are sums of idempotent and $q$-potent matrices.
\newblock {\em Siberian Math. J.}, {\bf 62}(1):1--13, 2021.

\bibitem{AT2}
A.~N. Abyzov, D.~T. Tapkin.
\newblock When is every matrix over a ring the sum of two tripotents?
\newblock {\em Linear Algebra \& Appl.} {\bf 630}:316--325, 2021.

\bibitem{B}
S.~Breaz.
\newblock Matrices over finite fields as sums of periodic and nilpotent elements.
\newblock {\em Linear Algebra \& Appl.}, {\bf 555}:92--97, 2018.

\bibitem{BCDM}
S.~Breaz, G.~C\v{a}lug\v{a}reanu, P.~Danchev and T.~Micu.
\newblock Nil-clean matrix rings.
\newblock {\em Linear Algebra \& Appl.}, {\bf 439}:3115--3119, 2013.

\bibitem{BM}
S.~Breaz and S.~Megiesan.
\newblock Nonderogatory matrices as sums of idempotent and nilpotent matrices.
\newblock {\em Linear Algebra \& Appl.}, {\bf 605}:239--248, 2020.

\bibitem{CL}
G.~Calugareanu and T.~Y. Lam.
\newblock Fine rings: A new class of simple rings.
\newblock {\em J. Algebra \& Appl.}, {\bf 15}(9), 2016.

\bibitem{CZ1}
G.~C\v{a}lug\v{a}reanu and Y.~Zhou.
\newblock Rings with fine idempotents.
\newblock {\em J. Algebra \& Appl.}. {\bf 21}(1):2250013 (14 pages), 2022.

\bibitem{CZ}
G.~C\v{a}lug\v{a}reanu and Y.~Zhou.
\newblock Rings with fine nilpotents.
\newblock {\em Ann. Univ. Ferrara Sez. VII Sci. Mat.}, {\bf 67}:231–241, 2021.

\bibitem{Co}
P.~M. Cohn
\newblock{The similarity reduction of matrices over a skew field.}
\newblock{\em Math. Z.}, {\bf 132}:151--163, 1973.

\bibitem{D}
P.~Danchev.
\newblock On some decompositions of matrices over algebraically closed and finite fields.
\newblock {\em J. Siberian Federal Univ. -- Math. \& Phys.}, {\bf 14}(5):547--553, 2021.

\bibitem{Da}
P.~V.~Danchev.
\newblock Nil-good unital rings.
\newblock {\em Internat. J. Algebra}, {\bf 10}(5):239-252, 2016.

\bibitem{DGL2}
P.~Danchev, E.~Garc\'ia and M.~G. Lozano.
\newblock Decompositions of matrices into potent and square-zero matrices.
\newblock {\em Internat. J. Algebra \& Computat.}, {\bf 32}(2):251--263, 2022.

\bibitem{DGL1}
P.~Danchev, E.~Garc\'ia and M.~G\'omez Lozano.
\newblock Decompositions of matrices into diagonalizable and square-zero matrices.
\newblock {\em Linear \& Multilinear Algebra}, {\bf 71}(3), 2023.

\bibitem{DGL3}
P.~Danchev, E.~Garc\'ia and M.~G\'omez Lozano.
\newblock Decompositions of matrices into a sum of invertible matrices and matrices of fixed nilpotence.
\newblock {\em arXiv}: 2302.12751.

\bibitem{GGVdSMA1}
E.  ~Garc\'ia, M.~G\'omez Lozano, R. Mu\~{n}oz Alc\'azar and G. Vera de Salas.
\newblock{A Jordan canonical form for nilpotent elements in an arbitrary
ring.}
\newblock{\em Linear Algebra \& Appl.}, {\bf 581}: 1--12, 2019.

\bibitem{GGVdSMA2}
E.  ~Garc\'ia, M.~G\'omez Lozano, R. Mu\~{n}oz Alc\'azar and G. Vera de Salas.
\newblock{Gradings induced by nilpotent elements.}
\newblock{\em Linear Algebra \& Appl.}, {\bf 656}: 1--20, 2023.

\bibitem{GS}
A.~B.~Gorman and W.~Y.~Shiao.
\newblock Nil-good and nil-good clean matrix rings.
\newblock {\em arXiv}: 1512.04640v1.

\bibitem{H}
T.~ W. Hungerford.
\newblock {\em Algebra}.
\newblock Volume {\bf 73} of Graduate Texts in Mathematics, Springer-Verlag, New York-Berlin, 1980. (Reprint of the 1974 original.)

\bibitem{Sh}
Y.~Shitov.
\newblock The ring $\mathbb{M}_{8k+4}(\mathbb{Z}_2)$ is nil-clean of index four.
\newblock {\em Indag. Math. (N.S.)}, {\bf 30}:1077-1078, 2019.

\bibitem{Sj}
D. ~Sjerve.
\newblock Canonical forms for torsion matrices.
\newblock{\em J. Pure Appl. Algebra}, {\bf 22}:103--111, 1981.

\bibitem{St1}
J.~\v{S}ter.
\newblock On expressing matrices over $\mathbb{Z}_2$ as the sum of an idempotent and a nilpotent.
\newblock {\em Linear Algebra \& Appl.}, {\bf 544}:339--349, 2018.

\bibitem{St2}
J.~\v{S}ter.
\newblock Nil-clean index of $\mathbb{M}_n(\mathbb{F}_2)$.
\newblock {\em Linear Algebra \& Appl.}, {\bf 632}(1):294--307, 2022.

\end{thebibliography}

\end{document}